\documentclass[11pt,a4paper]{amsart}
\usepackage{amsfonts}
\usepackage{amsthm}
\usepackage{amsmath}
\usepackage{amscd}
\usepackage[latin2]{inputenc}
\usepackage{t1enc}
\usepackage[mathscr]{eucal}
\usepackage{indentfirst}
\usepackage{graphicx}
\usepackage{graphics}
\usepackage{pict2e}
\usepackage{epic}
\usepackage{lipsum}
\usepackage{epstopdf}

\usepackage{amssymb}
\usepackage{amsfonts}
\usepackage{tikz}
\usepackage{amssymb}
\usepackage{amsthm}
\usepackage{arydshln}
\usepackage{newlfont}
\usepackage{wrapfig}
\usepackage{graphicx}
\input amssym.def
\input amssym
\def\dj{d\kern-0.4em\char"16\kern-0.1em}

 \newtheorem{thm}{Theorem}
 \newtheorem{cor}[thm]{Corollary}
 \newtheorem{lem}[thm]{Lemma}
 \newtheorem{prop}[thm]{Proposition}
\theoremstyle{definition}
 \newtheorem{defn}[thm]{Definition}
\newtheorem{exm}[thm]{Example}
 \newtheorem{rem}[thm]{Remark}

\title{The cubical matching complex revisited}
{} 

\author{Du\v sko Joji\' c}
\address{University of Banja Luka, Faculty of Science, \\Mladena
Stojanovi\' ca 2, 78 000 Banja Luka, Bosnia and Herzegovina}
\email{ducci68@teol.net}

\begin{document}
\maketitle

\begin{abstract}

Ehrenborg noted that all tilings of a bipartite planar graph are
encoded by its cubical matching complex and claimed that this
complex is collapsible. We point out to an oversight in his proof
and explain why these complexes can be the union of collapsible
complexes. Also, we prove that all links in these complexes are
suspensions up to homotopy. Furthermore, we extend the definition
of a cubical matching complex to planar graphs that are not
necessarily bipartite, and show that these complexes are either
contractible or a disjoint union of contractible complexes.

For a simple connected region that can be tiled with dominoes
($2\times 1$ and $1\times 2$) and $2\times 2$ squares, let $f_i$
denote the number of tilings with exactly $i$ squares. We prove
that $f_0-f_1+f_2-f_3+\cdots=1$ (established by Ehrenborg) is the
only linear relation for the numbers $f_i$.

\end{abstract}



\section{Introduction}
Let $G=(V,E)$ be a bipartite planar graph that allows a perfect
matching. Assume that $G$ is embedded in a plane. An
\textit{elementary cycle} of $G$ is a cycle that encircles a
single region $R$ different than outer region $R^*$. Throughout
this paper, we identify an elementary cycle with the region it
encircles as well as with its set of vertices or edges.

A \textit{tiling} of $G$ is a partition of the vertex set $V$ into
 disjoint blocks of the following two types:
\begin{itemize}
\item[(1)] an edge $\{x,y\}$ of $G$; or
 \item[(2)] an elementary cycle $R$ (the set of vertices
 of $R$).
\end{itemize}

The set of all tilings of $G$ form a cubical complex
$\mathcal{C}(G)$ (called the \textit{cubical matching complex})
defined by Ehrenborg in \cite{EhrCUB}. Note that $\mathcal{C}(G)$
depends not only on $G$, but also on the choice of the embedding
of that graph in the plane.

A face $F$ of $\mathcal{C}(G)$ has the form $F=M_F \cup
C_F=(M_F,C_F),$ where $C_F$ is a collection
$C_F=\{R_1,R_2,\ldots,R_t\}$ of vertex-disjoint elementary cycles
of $G$, and $M_F$ is a perfect matching on $G\setminus
\big(R_1\cup R_2 \cup \cdots\cup R_t\big)$. The dimension of $F$
is $|C_F|$, and the vertices of $\mathcal{C}(G)$ are all perfect
matchings of $G$.

 All tilings of $G$ covered by $F=(M_F,C_F)$ can be
obtained by deleting an elementary cycle $R$ from $C_F$, and
adding every other edge of $R$ into $M_F$ (there are two
possibilities to do this). Therefore, for two faces $F_1=(M_{F_1},
C_{F_1})$ and $F_2=(M_{F_2}, C_{F_2})$, we have that
\begin{equation}\label{E:facerelation}
\big(F_1\subset F_2\big) \Longleftrightarrow \Big(C_{F_1}\subset
C_{F_2}\textrm{ and }
 M_{F_1}\supset M_{F_2}\Big).
\end{equation}

Let $G^\circ$ denote the weak dual graph
 of a planar graph
$G$. The vertices of $G^\circ$ are all bounded regions of $G$, and
two regions that share a common edge are adjacent in $G^\circ$.

The \textit{independence complex} of a graph $H$ is a simplicial
complex $I(H)$ whose faces are the independent subsets of vertices
of $H$. Note that for any face $F= (M_F,C_F)$ of $\mathcal{C}(G)$,
the set $C_F$ contains independent vertices of $G^\circ$, i.e.,
$C_F$ is a face of $I(G^\circ)$.

At the first sight, the complex $\mathcal{C}(G)$ is related with
the independence complex $I(G^\circ)$ of its weak dual graph.

\begin{figure}[htbp]
\begin{center}
\begin{tikzpicture}[scale=.96]{center}
\draw (0,0) -- (1,0) -- (1,1) -- (0,1) -- (0,0); \draw (2,0) --
(3,0) -- (3,1) -- (2,1) -- (2,0); \draw (1,1) -- (1.5,1.7) --
(2,1);\draw (1,0) -- (1.5,-0.7) -- (2,0);

 \draw [fill] (0,0)
circle [radius=0.051];\draw [fill] (0,1) circle [radius=0.051];
\draw [fill] (1,0) circle [radius=0.051]; \draw [fill] (1,1)
circle [radius=0.051];

 \draw [fill] (2,0)
circle [radius=0.051];\draw [fill] (2,1) circle [radius=0.051];
\draw [fill] (3,0) circle [radius=0.051]; \draw [fill] (3,1)
circle [radius=0.051];

 \draw [fill] (1.5,1.7)
circle [radius=0.051];\draw [fill] (1.5,-0.7) circle
[radius=0.051]; \node at (.5,.5) {A};\node at (1.5,.5) {B};\node
at (2.5,.5) {C};

\node at (1.7,-1.2) {$G_1$};\node at (5.6,-1.2) {$G_2$};\node at
(10,-1.2) {$G_3$};

\node at (1.7,-3.2) {$\mathcal{C}(G_1)$};\node at (6.68,-3.2)
{$\mathcal{C}(G_2)$};\node at (10.7,-3.2) {$\mathcal{C}(G_3)$};

\draw (4,0) -- (5,0) -- (5,1) -- (4,1) -- (4,0); \draw (6,0) --
(7,0) -- (7,1) -- (6,1) -- (6,0); \draw (5,1) -- (6,1);\draw (5,0)
-- (5.33,-0.7) -- (5.66,-0.7)--(6,0);

 \draw [fill] (4,0)
circle [radius=0.051];\draw [fill] (4,1) circle [radius=0.051];
\draw [fill] (5,0) circle [radius=0.051]; \draw [fill] (5,1)
circle [radius=0.051];

 \draw [fill] (6,0)
circle [radius=0.051];\draw [fill] (6,1) circle [radius=0.051];
\draw [fill] (7,0) circle [radius=0.051]; \draw [fill] (7,1)
circle [radius=0.051];

 \draw [fill] (5.33,-.7)
circle [radius=0.051];\draw [fill] (5.66,-0.7) circle
[radius=0.051]; \node at (4.5,.5) {A};\node at (5.5,.5) {B};\node
at (6.5,.5) {C};

\draw (8,0.5) -- (8.7,1.2) -- (9.4,0.5) -- (8.7,-0.2) -- (8,0.5);
\draw (8.7,1.2) -- (10.7,1.2); \draw (10.7,-0.2)--(8.7,-0.2);

 \draw (10,0.5) -- (10.7,1.2) -- (11.4,0.5) -- (10.7,-0.2) --
(10,0.5);

\draw [fill] (11.4,0.5) circle [radius=0.051];\draw [fill]
(10.7,-0.2) circle [radius=0.051];

 \draw [fill] (8,0.5)
circle [radius=0.051];\draw [fill] (8.7,1.2) circle
[radius=0.051]; \draw [fill] (9.4,0.5) circle [radius=0.051];
\draw [fill] (8.7,-0.2) circle [radius=0.051];

 \draw [fill] (10,0.5)
circle [radius=0.051];\draw [fill] (10.7,1.2) circle
[radius=0.051]; \node at (8.75,.5) {$A$};\node at (9.65,.25)
{$B$};\node at (10.75,.5) {$C$};

\draw [ultra thick] (0.2,-3) -- (1.,-4)--(2.,-4)--(2.8,-3); \draw
[ultra thick] [fill=gray!19] (4.5,-2.5) rectangle (6,-4);,\draw
[ultra thick](6,-4)--(7.2,-4); \node at (5.3,-3.3) {$AC$};\node at
(6.6,-4.27) {$B$};

\node at (0.34,-3.77) {$A$};\node at (1.5,-4.27) {$B$};\node at
(2.6,-3.77) {$C$}; \draw [ultra thick] [fill=gray!19] (8.5,-2.5)
rectangle (10,-4);\node at (9.3,-3.3) {$AC$};
\end{tikzpicture}
\caption{The three graphs with the same weak dual, but different
cubical matching complexes.} \label{F:istiIdrugiC}
\end{center}
\end{figure}
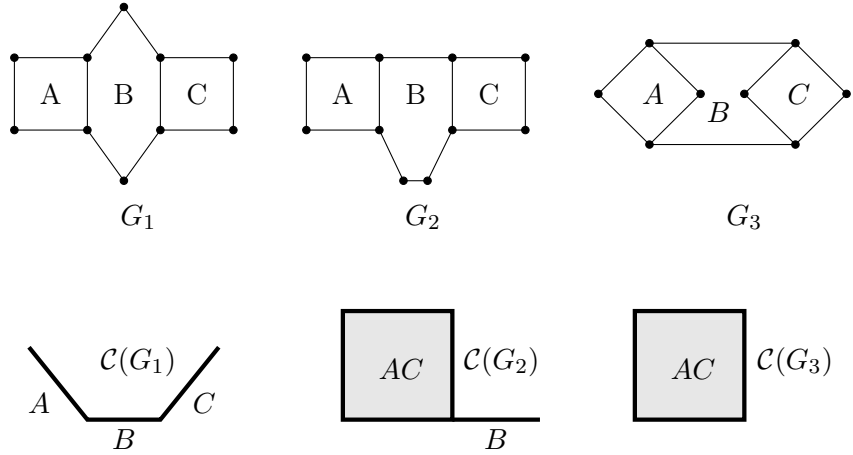
However, Figure \ref{F:istiIdrugiC} shows the three graphs with
the same weak dual but different cubical matching complexes. The
facets of the complexes on Figure \ref{F:istiIdrugiC} are labeled
by corresponding subsets of pairwise disjoint elementary regions.
\begin{exm}Let
$\mathcal{L}_n$ and $\mathcal{C}_n$ denote the independence
complexes of $P_n$ and $C_n$ (the path and cycle with $n$
vertices) respectively. The homotopy types of these complexes are
determined by Kozlov in \cite{Koz}:

$$\mathcal{L}_n \simeq\left\{%
\begin{array}{ll}
 \textrm{ a point } , & \hbox{if $n=3k+1$;} \\
    S^{\lfloor\frac{n-1}{3}\rfloor}, & \hbox{otherwise.} \\
\end{array}%
\right.   \mathcal{C}_n \simeq\left\{%
\begin{array}{ll}
    S^{k-1}, & \hbox{if $n=3k\pm 1$;} \\
S^{k-1}\vee S^{k-1}, & \hbox{if $n=3k$.} \\
\end{array}%
\right.  $$ We will use these complexes later, see Corollary
\ref{C:allLorC} and Remark \ref{R:OQ}. More details about
combinatorial and topological properties of $\mathcal{L}_n$ and
$\mathcal{C}_n$ (and about the independence complexes in general),
an interested reader can find in \cite{Eh-He}, \cite{Eng} and
\cite{Jonss}.
\end{exm}

There are some cubical complexes that cannot be realized as
subcomplexes of a $d$-cube $C^d=[0,1]^d$, see Chapter $4$ of
\cite{tortop}.
 \begin{prop}\label{P:cord} Let $G$ be a bipartite planar graph
 that has a
 perfect matching. If $G$ has $d$ elementary regions, then its
 cubical matching complex $\mathcal{C}(G)$
 can be embedded into $C^d$.
 \end{prop}

\begin{proof}
  We use an idea from \cite{Propp} to describe the coordinates of
  vertices of $\mathcal{C}(G)$ explicitly.
Let $R_1,R_2,\ldots,R_d$ be a fixed linear order of elementary
regions of $G$. We choose an arbitrary perfect matching $M_0$ of
$G$ (a vertex of $\mathcal{C}(G)$) to be the origin
$\mathbf{0}=(0,0,\ldots,0)$ in $\mathbb{R}^d$. For another vertex
$M$ of $\mathcal{C}(G)$, we consider the symmetric difference $
M\triangle M_0$. Note that $ M \triangle M_0$ is a disjoint union
of cycles. For a given perfect matching $M$ of $G$, we assign the
vertex $V_M=(x_1,\ldots ,x_d)$ of $C^d$, where
$$x_i=\left\{%
\begin{array}{ll}
    1, & \hbox{if $R_i$ is contained into
    an odd number of cycles of $ M\triangle M_0$;} \\
    0, & \hbox{otherwise.} \\
\end{array}%
\right.    $$

If $M'$ and $M''$ are two perfect matchings of $G$ such that
$M'\triangle M''=R_j$ (meaning that these two matchings differ
just on an elementary region $R_j$), then their corresponding
vertices $V_{M'}$ and $V_{M''}$ of $C^d$ differ only at the $j$-th
coordinate.

Therefore, the face $F=(M_F,C_F)$ is embedded in $C^d$ as the
convex hull of its $2^{|C(F)|}$ vertices.

\end{proof}

\section{The local structure of $\mathcal{C}(G)$}
The \textit{star} of a face $F$ in a cubical complex $\mathcal{C}$
is the set of all faces of $\mathcal{C}$ that contain $F$
 $$star(F)=\{F'\in
\mathcal{C}: F\subset F'\}.$$ The \textit{link} of a vertex $v$ in
a cubical complex $\mathcal{C}$ is the simplicial complex
$link_\mathcal{C}(v)$ that can be realized in $\mathcal{C}$ as a
``small sphere`` around the vertex $v$. More formally, the
vertices of $link_\mathcal{C}(v)$ are the edges of $\mathcal{C}$
containing $v$. A subset of vertices of $link_\mathcal{C}(v)$ is a
face of $link_\mathcal{C}(v)$ if and only if the corresponding
edges belong to a same face of $\mathcal{C}$.

 The \textit{link} of a face $F$ in
a cubical complex $\mathcal{C}$ is defined in a similar way. The
set of vertices of $link_\mathcal{C}(F)$ is $$\{F'\in \mathcal{C}:
F\subset F'\textrm{ and }dim\, F'=1+dim\, F\},$$ and a subset $A$
of the set of vertices is a face of $link_\mathcal{C}(F)$ if and
only if all elements of $A$ are contained in a same face of
$\mathcal{C}$.

Ehrenborg investigated the links of the cubical complexes
associated to tilings of a region by dominos or lozenges.

Here we describe the links in the cubical matching complex
$\mathcal{C}(G)$ for any bipartite planar graph $G$. For a face
$F=(M_F,C_F)$ of $\mathcal{C}(G)$, let $\mathcal{R}_F$ denote the
set of all elementary regions of $G$ for which every second edge
is contained in $M_F$. Further, let $G_F$ denote the subgraph of
the weak dual graph $G^\circ$ spanned with the regions from
$\mathcal{R}_F$.

From the definition of the link in a cubical complex and
(\ref{E:facerelation}), we obtain the next statement.
\begin{prop}
For any face $F=(M_F,C_F)$ of $\mathcal{C}(G)$ we have that
$$link_\mathcal{C}(F)\cong I(G_F).$$
\end{prop}

The above proposition explains the appearance of complexes
$\mathcal{L}_n$ and $\mathcal{C}_n$ as the links in cubical the
matching complexes, see Theorem 3.3 and Section 4 in
\cite{EhrCUB}.

Assume that all elementary regions of $G$ are quadrilaterals. In
that case, for any face $F$ of $\mathcal{C}(G)$, the degree of a
vertex in $G_F$ is at most two. Therefore, $G_F$ is a union of
paths and cycles.
\begin{cor}\label{C:allLorC}
If all elementary regions of $G$ are quadrilaterals, then
$link_\mathcal{C}(F)$ is a join of complexes $\mathcal{L}_p$ and
$\mathcal{C}_{2q}$.
\end{cor}

\begin{thm} Let $G$ be a bipartite planar graph that has a
 perfect matching. For any face $F=(M_F,C_F)$ of
  $\mathcal{C}(G)$ the graph $G_F$ is
bipartite.
\end{thm}

\begin{proof}
Assume that $G_F$ contains an odd cycle $R_1,R_2,\ldots,
R_{2m+1}$. Recall that $R_i$ is an elementary region of $G$ and
the that every second edge of $R_i$ is contained in $M_F$. Two
neighborly regions $R_i$ and $R_{i+1}$ have to share the odd
number of edges, the first and the last of their common edges
belong to $M_F$. Therefore, for each region $R_i$, there is an odd
number of common edges of $R_i$ and $R_{i-1}$ that belong to
$M_F$. Obviously, the same holds for $R_i$ and $R_{i+1}$.

So, we can conclude that there is an odd number of edges
 of $R_i$ that are between $R_i \cap R_{i-1}$ and $R_i \cap R_{i+1}$
 (the first and the
 last one of these edges are not in $M_F$). The union of all
of these edges (for all regions $R_i$) is an odd cycle in $G$,
which is a contradiction.

\end{proof}
Barmak proved in \cite{Barmak} (see also in \cite{Na-Re}) that the
independence complexes of bipartite graphs are suspensions, up to
homotopy. This implies the next result.

\begin{cor}\label{Cor:link}
All links in $\mathcal{C}(G)$ are homotopy equivalent to
suspensions. Therefore, the link of any face in $\mathcal{C}(G)$
has at most two connected components.
\end{cor}

For any simplicial complex $K$ there exists a bipartite graph $G$
such that the independence complex of $G$ is homotopy equivalent
to the suspension over $K$, see \cite{Barmak}. Skwarski proved in
\cite{Skv} (see also \cite{Barmak}) that there exists a planar
graph $G$ whose independence complex is homotopy equivalent to an
iterated suspension of $K$.

We prove that the links of faces in cubical matching complexes are
independence complexes of bipartite planar graphs. What can be
said about homotopy types of these complexes?

\begin{rem}\label{R:OQ}
There is a natural question, posed by Ehrenborg in \cite{EhrCUB}:
\textit{For what graphs $G$ would the cubical matching complex
$\mathcal{C}(G)$ be pure, shellable, non-pure shellable?}

The complexes $\mathcal{L}_n$ are non-pure for $n>4$, and the
complexes $\mathcal{C}_n$ are non-shellable for $n>5$. Therefore,
these complexes can be used to show that the cubical matching
complex of a concrete graph is non-pure or non-shellable.

\end{rem}
\section{Collapsibility and contractibility of cubical
matching complexes}

 The next theorem is the main result in
\cite{EhrCUB}.

\begin{thm}[Theorem 1.2 in \cite{EhrCUB}]\label{T:EHr}
For a planar bipartite graph $G$ that has a perfect matching, the
cubical matching complex $\mathcal{C}(G)$ is collapsible.
\end{thm}
The proof of the above statement is based on the next two results:
\begin{itemize}
    \item[$(i)$](Propp, Theorem 2
in \cite{Propp}) \textit{The set of all perfect matchings of a
bipartite planar graph is a distributive lattice.}
    \item[$(ii)$](Kalai, see in
\cite{EC1}, Solution to Exercise 3.47
    c) \textit{The cubical
complex of a meet-distributive lattice is collapsible.}
\end{itemize}
\noindent Note however that Propp in his proof of $(i)$ assumed
the following two additional conditions for bipartite planar graph
$G$:
\begin{itemize}
    \item[$(*)$]  Graph $G$ is connected, and
    \item[$(**)$] Any edge of $G$ is contained in some matching of $G$
but not in others.
\end{itemize}

\begin{exm}\label{E:counter}The next figure shows a bipartite
planar graph whose cubical matching complex is not collapsible.

\begin{figure}[htbp]
\begin{center}
\begin{tikzpicture}[scale=.898]{center}
\draw [thick] (0,0) rectangle (4,4); \draw [thick] (1,1) rectangle
(3,3); \draw [fill] (0,0) circle [radius=0.111];\draw [fill] (0,4)
circle [radius=0.111];\draw [fill] (4,0) circle
[radius=0.111];\draw [fill] (4,4) circle [radius=0.111];

\draw [fill] (1,1) circle [radius=0.111];\draw [fill] (1,3) circle
[radius=0.111];\draw [fill] (3,1) circle [radius=0.111];\draw
[fill] (3,3) circle [radius=0.111]; \draw[thick] (0,0)--(1,1);
\draw[thick] (3,3)--(4,4);

\node at (2,-.5) {$G$};

\node at (9.7,-0.2) {$\mathcal{C}(G)$};

\draw [ultra thin] (6.6,3.5) rectangle (7.6,4.5); \draw [ultra
thin] (6.9,3.8) rectangle (7.3,4.2); \draw [ultra thin]
(6.6,3.5)--(6.9,3.8);\draw [ultra thin] (7.6,4.5)--(7.3,4.2);

\draw [ultra thin] (11.6,3.5) rectangle (12.6,4.5); \draw [ultra
thin] (11.9,3.8) rectangle (12.3,4.2);

\draw [ultra thin] (11.6,3.5)--(11.9,3.8);\draw [ultra thin]
(12.6,4.5)--(12.3,4.2);

\draw [ultra thin] (6.6,-.5) rectangle (7.6,0.5); \draw [ultra
thin] (6.9,-.2) rectangle (7.3,.2); \draw [ultra thin]
(6.6,-.5)--(6.9,-0.2);\draw [ultra thin] (7.6,0.5)--(7.3,0.2);

\draw [ultra thin] (11.6,-.5) rectangle (12.6,0.5); \draw [ultra
thin] (11.9,-.2) rectangle (12.3,.2); \draw [ultra thin]
(11.6,-.5)--(11.9,-0.2);\draw [ultra thin] (12.6,0.5)--(12.3,0.2);

\draw [ultra thick](6.6,3.5)--(6.6,4.5);\draw [ultra
thick](7.6,3.5)--(7.6,4.5); \draw [ultra
thick](7.3,3.8)--(7.3,4.2);\draw [ultra
thick](6.9,3.8)--(6.9,4.2);

\draw [ultra thick](11.6,3.5)--(11.6,4.5);\draw [ultra
thick](12.6,3.5)--(12.6,4.5); \draw [ultra
thick](11.9,-.2)--(12.3,-.2);\draw [ultra
thick](11.9,.2)--(12.3,.2);

\draw [ultra thick](6.6,-.5)--(7.6,-.5);\draw [ultra
thick](6.6,0.5)--(7.6,0.5); \draw [ultra
thick](12.3,3.8)--(11.9,3.8);\draw [ultra
thick](11.9,4.2)--(12.3,4.2);

\draw [ultra thick](11.6,-.5)--(12.6,-.5);\draw [ultra
thick](11.6,0.5)--(12.6,0.5);\draw [ultra
thick](7.3,.2)--(7.3,-.2);\draw [ultra thick](6.9,.2)--(6.9,-.2);

\draw [ultra thick] (8,0.5)--(11,0.5); \draw [ultra thick] (8,3.5)
--(11,3.5);
\end{tikzpicture}
\caption{A bipartite planar graph $G$ for which $\mathcal{C}(G)$
is not collapsible.} \label{F:nijekol}
\end{center}
\end{figure}
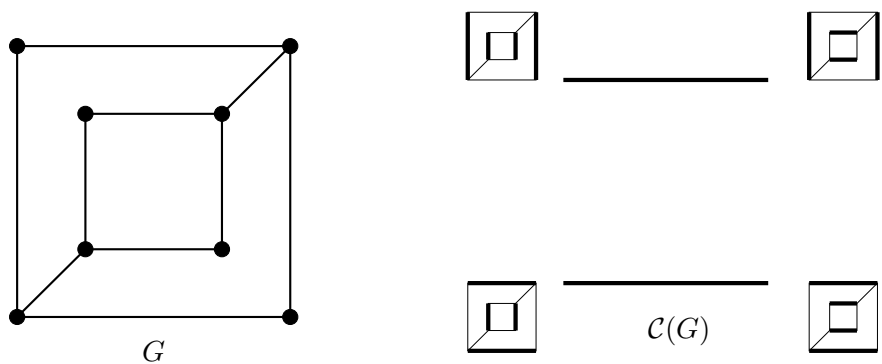

\end{exm}

 Also, the Jockusch
example (page 41 in \cite{Propp}, a bipartite planar graph with
$20$ edges, but just $12$ of them can be used in a perfect
matching), describe a graph $G$ whose cubical matching complex is
a disjoint union of four segments.

The edges that do not appear in any perfect matching of a graph
$G$ (the forbidden edges) can be deleted. Also, if the edge $xy$
is a forced edge ($xy$ appears in all perfect matching of $G$),
then we may consider the graph $G-\{x,y\}$.

\begin{figure}[htbp]
\begin{center}
\begin{tikzpicture}[scale=.7]{center}
\draw [fill=black] (0,.5) circle [radius=0.039]; \draw
[fill=black] (-0.5,4) circle [radius=0.039]; \draw [fill=black]
(1,4.8) circle [radius=0.039];\draw [fill=black] (2,0) circle
[radius=0.039];\draw [fill=black] (2.6,5) circle [radius=0.039];
\draw [fill=black] (4,.5) circle [radius=0.039]; \draw
[fill=black] (4,4.7) circle [radius=0.039]; \draw [fill=black]
(5,3) circle [radius=0.039];

\draw[thick]
(0,.5)--(-.5,4)--(1,4.8)--(2.6,5)--(4,4.7)--(5,3)--(4,.5)--(2,0)--(0,.5);

\draw[dashed](2,0)--(1,4.8);

\node at (1.8,2.8) {$e$}; \node at (7.8,2.8) {$e$};\node at
(13.8,2.8) {$e$};

\draw [fill=black] (6,.5) circle [radius=0.039]; \draw
[fill=black] (5.5,4) circle [radius=0.039]; \draw [fill=black]
(7,4.8) circle [radius=0.039];\draw [fill=black] (8,0) circle
[radius=0.039];\draw [fill=black] (8.6,5) circle [radius=0.039];
\draw [fill=black] (10,.5) circle [radius=0.039]; \draw
[fill=black] (10,4.7) circle [radius=0.039]; \draw [fill=black]
(11,3) circle [radius=0.039];

\draw[thick] (6,.5)--(5.5,4); \draw[dashed](5.5,4)--(7,4.8);
\draw[thick] (7,4.8)--(8.6,5); \draw[dashed](8.6,5)--(10,4.7);
\draw[thick](10,4.7)--(11,3);\draw[dashed] (11,3)--(10,.5);
\draw[thick](10,.5)--(8,0);\draw[dashed] (8,0)--(6,.5);

\draw[dashed](8,0)--(7,4.8);

\draw [fill=black] (12,.5) circle [radius=0.039]; \draw
[fill=black] (11.5,4) circle [radius=0.039]; \draw [fill=black]
(13,4.8) circle [radius=0.039];\draw [fill=black] (14,0) circle
[radius=0.039];\draw [fill=black] (14.6,5) circle [radius=0.039];
\draw [fill=black] (16,.5) circle [radius=0.039]; \draw
[fill=black] (16,4.7) circle [radius=0.039]; \draw [fill=black]
(17,3) circle [radius=0.039];

\draw[thick] (12,.5)--(11.5,4); \draw[dashed](11.5,4)--(13,4.8);
\draw[dashed] (13,4.8)--(14.6,5); \draw[thick](14.6,5)--(16,4.7);
\draw[dashed](16,4.7)--(17,3);\draw[thick] (17,3)--(16,.5);
\draw[dashed](16,.5)--(14,0);\draw[dashed] (14,0)--(12,.5);

\draw[thick](14,0)--(13,4.8);

\end{tikzpicture}
\caption{If a new region can be included in a tiling of $G-e$,
then $e$ is not forbidden.} \label{F:GG'}
\end{center}
\end{figure}
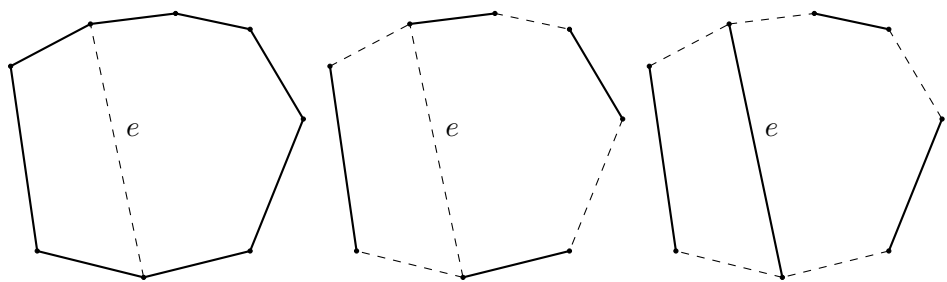

\begin{rem}\label{R:GsameG'}
Let $e$ denote a forbidden edge in $G$ and let $G'=G-e$. The
possible new elementary region of $G'$, that appears after we
delete $e$, can not be included in a tiling of $G'$. Otherwise, we
can find a perfect matching of $G$ that contains $e$, see Figure
\ref{F:GG'}. In a similar way we conclude that the new regions
that appear after deleting a forced edge can not be included in a
tiling of $G'$.
\end{rem}

Let $G'$ denote the graph obtained from $G$ after all deletions.
Unfortunately, this new graph (after deleting all forced and
forbidden edges) may be non-connected.

If $G'$ is connected, then the collapsibility of $\mathcal{C}(G')$
follows from Ehrenborg's proof. Also, if $G'$ is non-connected,
and all of its connected components are separated (there is no
component of $G'$ that is contained in an elementary region of
another component), then $\mathcal{C}(G')$ is collapsible as a
product of collapsible complexes.

By using Remark \ref{R:GsameG'}, we can establish an obvious
bijection between tilings of $G'$ and tilings of $G$ (we just add
all forced edges). Therefore, Theorem \ref{T:EHr} holds if $G'$ is
connected or if all of its connected components are separated.

However, Theorem \ref{T:EHr} fails if $G'$ has two different
connected components $G_1$ and $G_2$ such that $G_1$ is contained
in an elementary region $R$ of $G_2$, see Example \ref{E:counter}.
In that case we have that
$$\mathcal{C}(G')=\mathcal{C}(G_1)\times\left(\mathcal{C}(G_2)
\setminus\{R\}\right),$$and $\mathcal{C}(G')$ is a union of
collapsible complexes. Here $\mathcal{C}(G_2) \setminus\{R\}$
denote the cubical complex obtained from $\mathcal{C}(G_2)$ by
deleting all tilings (faces) that contain $R$ as an elementary
region.

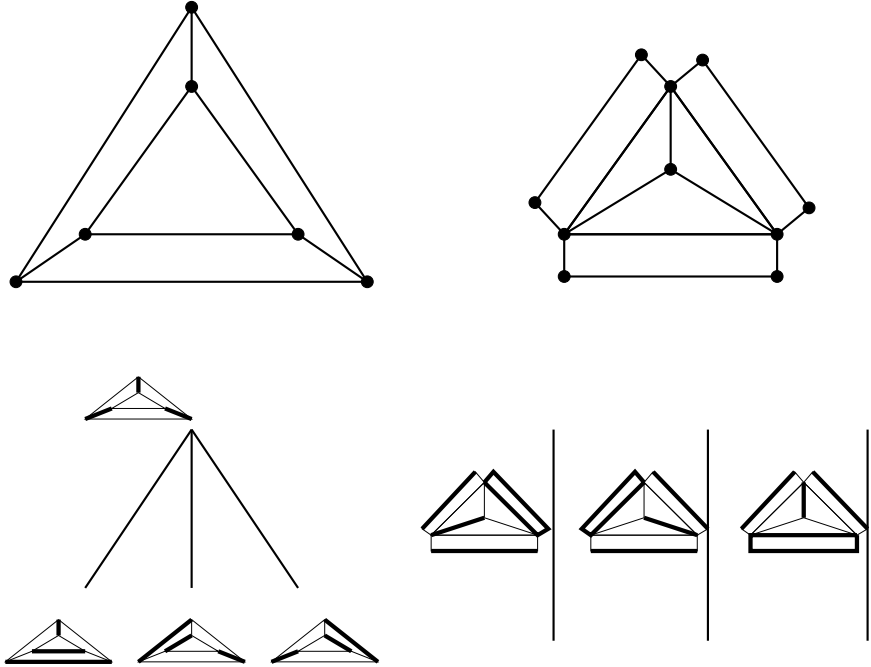
\begin{figure}[htbp]
\begin{center}
\begin{tikzpicture}[scale=.7]{center}
\draw [thick] (0,4.5)--(2,1.7)--(-2,1.7)--(0,4.5); \draw [ thick]
(0,6)--(3.3,.8)--(-3.3,.8)--(0,6);  \draw [ thick] (0,6)--(0,4.5);

\draw [ thick] (3.3,.8)--(2,1.7); \draw [ thick]
(-3.3,.8)--(-2,1.7);

\draw [fill] (0,4.5) circle [radius=0.111]; \draw [fill] (2,1.7)
circle [radius=0.111]; \draw [fill] (-2,1.7) circle
[radius=0.111]; \draw [fill] (0,6) circle [radius=0.111]; \draw
[fill] (3.3,.8) circle [radius=0.111]; \draw [fill] (-3.3,.8)
circle [radius=0.111];


\draw [fill] (9,4.5) circle [radius=0.111]; \draw [fill] (11,1.7)
circle [radius=0.111];

\draw [fill] (7,1.7) circle [radius=0.111]; \draw [fill] (9,2.93)
circle [radius=0.111];

\draw [ thick] (9,4.5)--(11,1.7)--(7,1.7)--(9,4.5);

\draw [thick] (9,2.93)--(9,4.5);\draw [ thick]
(9,2.93)--(11,1.7);\draw [thick] (9,2.93)--(7,1.7);

\draw [ thick] (9,4.5)--(11,1.7)--(11.6,2.2)--(9.6,5)--(9,4.5);

\draw [fill] (11.6,2.2) circle [radius=0.111]; \draw [fill]
(9.6,5) circle [radius=0.111];

 \draw [ thick]
(11,1.7)--(7,1.7)--(7,.9)--(11,.9)--(11,1.7);

\draw [fill] (7,.9) circle [radius=0.111]; \draw [fill] (11,.9)
circle [radius=0.111];

 \draw [
thick] (7,1.7)--(9,4.5)--(8.45,5.1)--(6.45,2.3)--(7,1.7); \draw
[fill] (8.45,5.1) circle [radius=0.111]; \draw [fill] (6.45,2.3)
circle [radius=0.111];

\draw [thick] (0,-2)--(-2,-5);\draw [thick] (0,-2)--(0,-5); \draw
[ thick] (0,-2)--(2,-5);

\draw [ultra thin] (-1,-1)--(0,-1.8)--(-2,-1.8)--(-1,-1);

\draw [ultra thin] (-1,-1.3)--(-0.5,-1.6)--(-1.5,-1.6)--(-1,-1.3);
\draw [ultra thick](-1,-1)--(-1,-1.3);\draw [ultra
thick](0,-1.8)--(-0.5,-1.6);\draw [ultra
thick](-2,-1.8)--(-1.5,-1.6);

\draw [ultra thin]
(-2.5,-5.6)--(-1.5,-6.4)--(-3.5,-6.4)--(-2.5,-5.6);

\draw [ultra thin] (-2.5,-5.9)--(-2,-6.2)--(-3,-6.2)--(-2.5,-5.9);
\draw [ultra thick](-2.5,-5.6)--(-2.5,-5.9);\draw [ultra
thick](-1.5,-6.4)--(-3.5,-6.4);\draw [ultra
thick](-2,-6.2)--(-3,-6.2); \draw [ultra
thin](-1.5,-6.4)--(-2,-6.2); \draw [ultra
thin](-3.5,-6.4)--(-3,-6.2);

\draw [ultra thin] (0,-5.6)--(1,-6.4)--(-1,-6.4)--(0,-5.6); \draw
[ultra thin] (0,-5.9)--(.5,-6.2)--(-.5,-6.2)--(0,-5.9); \draw
[ultra thick](1,-6.4)--(.5,-6.2); \draw [ultra
thick](0,-5.6)--(-1,-6.4); \draw [ultra
thick](-.5,-6.2)--(0,-5.9);

\draw [ultra thin] (0,-5.6)--(0,-5.9);\draw [ultra thin]
(-1,-6.4)--(-.5,-6.2);

\draw [ultra thin] (2.5,-5.6)--(3.5,-6.4)--(1.5,-6.4)--(2.5,-5.6);
\draw [ultra thin] (2.5,-5.9)--(3,-6.2)--(2,-6.2)--(2.5,-5.9);
\draw [ultra thick](1.5,-6.4)--(2,-6.2); \draw [ultra
thick](3.5,-6.4)--(2.5,-5.6); \draw [ultra
thick](2.5,-5.9)--(3,-6.2); \draw [ultra thin]
(2.5,-5.6)--(2.5,-5.9); \draw [ultra thin] (3.5,-6.4)--(3,-6.2);

\draw [thick] (9.7,-2)--(9.7,-6); \draw [thick]
(12.7,-2)--(12.7,-6); \draw [thick] (6.8,-2)--(6.8,-6);

\draw [ultra thin] (5.5,-3)--(6.5,-4)--(4.5,-4)--(5.5,-3);

\draw [ultra thin] (5.5,-3.67)--(5.5,-3);\draw [ultra thin]
(5.5,-3.67)--(6.5,-4);\draw [ultra thick] (5.5,-3.67)--(4.5,-4);

\draw [ultra thick]
(5.5,-3)--(6.5,-4)--(6.7,-3.88)--(5.67,-2.8)--(5.5,-3);

\draw [ultra thin]
(5.5,-3)--(4.5,-4)--(4.33,-3.88)--(5.33,-2.8)--(5.5,-3); \draw
[ultra thin] (6.5,-4)--(4.5,-4)--(4.5,-4.3)--(6.5,-4.3)--(6.5,-4);
\draw [ultra thick] (4.33,-3.88)--(5.33,-2.8); \draw [ultra thick]
(4.5,-4.3)--(6.5,-4.3);

\draw [ultra thin] (8.5,-3)--(9.5,-4)--(7.5,-4)--(8.5,-3);

\draw [ultra thin] (8.5,-3.67)--(8.5,-3);\draw [ultra thick]
(8.5,-3.67)--(9.5,-4);\draw [ultra thin] (8.5,-3.67)--(7.5,-4);

\draw [ultra thin]
(8.5,-3)--(9.5,-4)--(9.7,-3.88)--(8.67,-2.8)--(8.5,-3);

\draw [ultra thick]
(8.5,-3)--(7.5,-4)--(7.33,-3.88)--(8.33,-2.8)--(8.5,-3); \draw
[ultra thin] (9.5,-4)--(7.5,-4)--(7.5,-4.3)--(9.5,-4.3)--(9.5,-4);

 \draw [ultra thick]
(9.7,-3.88)--(8.67,-2.8); \draw [ultra thick]
(7.5,-4.3)--(9.5,-4.3);

\draw [ultra thin] (11.5,-3)--(12.5,-4)--(10.5,-4)--(11.5,-3);
\draw [ultra thick] (11.5,-3)--(11.5,-3.67); \draw [ultra thin]
(11.5,-3.67)--(11.5,-3);\draw [ultra thin]
(11.5,-3.67)--(12.5,-4);\draw [ultra thin]
(11.5,-3.67)--(10.5,-4);

\draw [ultra thin]
(11.5,-3)--(12.5,-4)--(12.7,-3.88)--(11.67,-2.8)--(11.5,-3); \draw
[ultra thick](12.7,-3.88)--(11.67,-2.8); \draw [ultra thick]
(10.33,-3.88)--(11.33,-2.8);
 \draw [ultra thin]
(11.5,-3)--(10.5,-4)--(10.33,-3.88)--(11.33,-2.8)--(11.5,-3);
\draw [ultra thick]
(12.5,-4)--(10.5,-4)--(10.5,-4.3)--(12.5,-4.3)--(12.5,-4);

\end{tikzpicture}
\caption{Non-bipartite graphs and their cubical matching
complexes.} \label{F:primjerigener}
\end{center}
\end{figure}

Now, we consider the cubical matching complex for all planar
graphs that have a perfect matching (not necessarily bipartite).

\begin{defn}Let $G$ be a planar graph that allows a perfect
matching. A tiling of $G$ is a partition of the vertex set $V$
into disjoint blocks of the following two types:
\begin{itemize}
\item  an edge $\{x,y\}$ of $G$; or
 \item the set of vertices $\{v_1,v_2,\ldots, v_{2m}\}$
 of an even elementary cycle $R$.
 \end{itemize} Let $\mathcal{C}(G)$ denote the set of
 all tilings of
$G$. Note that $\mathcal{C}(G)$ is also a cubical complex.
\end{defn}

\begin{exm}
If $G$ is a graph of a triangular prism (embedded in the plane so
that the outer region is a triangle), then $\mathcal{C}(G)$ is a
union of three $1$-dimensional segments that share the same
vertex, see the left side of Figure \ref{F:primjerigener}. Each of
segments of $\mathcal{C}(G)$ corresponds to a rectangle of prism.
The link of the common vertex of these segments is a
$0$-dimensional complex with three points. Such situation is no
possible for bipartite planar graphs, see Corollary
\ref{Cor:link}.
\end{exm}

The next theorem describe the homotopy type of the cubical
matching complex associated to a planar graph that allows a
perfect matching.
\begin{thm}\label{T:generalcase}Let $G$ be a planar graph that has
a perfect matching. The cubical complex $\mathcal{C}(G)$ is
contractible or a
 disjoint union of
contractible complexes.
\end{thm}
This is a weaker version (we prove contractibility instead
collapsibility) of corrected Theorem \ref{T:EHr}, with a different
proof.

\begin{proof}
We use the induction on the number of edges of $G$. Let $e=xy$
denote an edge that belongs to the outer region $R^*$. Let $R\neq
R^*$ denote the elementary region that contains $e$. If $R$ is an
odd region, then all tilings of $G$ can be divided into two
disjoint classes:
\begin{itemize}
    \item[(a)] The
tilings of $G$ that do not use $e$. These tilings are just the
tilings of $G\setminus e$.
    \item[(b)]  The
tilings of $G$ that contain $e$ as an edge in a partial matching
correspond to the tilings of $G\setminus\{x,y\}$.
\end{itemize}
In that case we obtain that
$\mathcal{C}(G)=\mathcal{C}(G\setminus\{x,y\})\sqcup
 \mathcal{C}(G\setminus e)$
is a disjoint union of contractible complexes by inductive
 assumption.
\\

\noindent If $R$ is an even elementary region, then some tilings
of $G$ may to contain $R$. Note that these tilings are not
considered in (a) and (b). To describe the corresponding faces of
$\mathcal{C}(G)$, we consider $G\setminus R$, the graph obtained
from $G$ by deleting all vertices from $R$.

 Let $\mathcal{C}_e$ denote the subcomplex of $
 \mathcal{C}(G\setminus e)$ formed by all tilings that contain every
 second edge of $R$ (but do not contain $e$, obviously). Further,
 let $\mathcal{C}_{x,y}$ denote the subcomplex of
$\mathcal{C}(G\setminus\{x,y\})$, defined by tilings that contain
every second edge of $R$ (these tilings have to contain $e$). Note
that the both of complexes $\mathcal{C}_e$ and $\mathcal{C}_{x,y}$
are isomorphic to
 $\mathcal{C}(G\setminus R)$. In that case we obtain
\begin{equation}\label{E:RRforC}
\mathcal{C}(G)=\mathcal{C}(G\setminus\{x,y\})\cup
 \mathcal{C}(G\setminus e)\cup
 Prism(\mathcal{C}(G\setminus R)).
\end{equation}
Further, we have that $$\mathcal{C}(G\setminus e)
 \cap Prism(\mathcal{C}(G\setminus R))=\mathcal{C}_e
\textrm{ and } \mathcal{C}(G\setminus \{x,y\})
 \cap Prism(\mathcal{C}(G\setminus R))=\mathcal{C}_{x,y}.$$

The complexes on the right-hand side of (\ref{E:RRforC}) are
disjoint unions of contractible complexes by the inductive
hypothesis. Assume that
$$ \mathcal{C}(G\setminus\{x,y\})=A_1 \sqcup A_2\sqcup\cdots
\sqcup A_s \textrm{ and }\mathcal{C}_{x,y}=B_1\sqcup B_2 \sqcup
\cdots \sqcup B_t,$$ where $A_i$ and $B_j$ denote the contractible
components of corresponding complexes. Obviously, each complex
$B_j$ is contained in some $A_i$. Now, we need the following
lemma.

\begin{lem}Each of connected component of
$\mathcal{C}(G\setminus\{x,y\})$ contains at most one component of
$\mathcal{C}_{x,y}$.
\end{lem}

\textit{Proof of Lemma: }Assume that a component of
$\mathcal{C}(G\setminus\{x,y\})$ contains two components of
$\mathcal{C}_{x,y}$. In that case, there are two vertices of
$\mathcal{C}_{x,y}$ (perfect matchings of $G$ that contain $xy$)
that are in different components of $\mathcal{C}_{x,y}$, but in
the same component of $\mathcal{C}(G\setminus\{x,y\})$. Assume
that $M'$ and $M''$ are two such vertices, chosen so that the
distance between them in $\mathcal{C}(G\setminus\{x,y\})$ is
minimal. Let
\begin{equation}\label{E:shorty}
M'=M_0 \frac{_{R_0}}{} M_1 \frac{\textrm{\,\,\,\, }}{}\ldots
\frac{\textrm{\,\,\,\, }}{}
 M_i \frac{_{R_i}}{}
M_{i+1} \frac{\textrm{\,\,\,\, }}{}\ldots \frac{\textrm{\,\,\,\,
}}{}M_{n}
 \frac{_{R_{n}}}{}M_{n+1}=M''
\end{equation}
 denote the shortest path from
$M'$ to $M''$ in $\mathcal{C}(G\setminus\{x,y\})$. The perfect
matching $M_{i+1}$ is obtained from $M_i$ by removing the edges of
$M_i$ contained in an elementary region $R_i$, and by inserting
the complementary edges. In other words, we have that
$M_{i+1}=M_i\triangle R_i$, for an elementary region $R_i$
contained in $\mathcal{R}_{F_i}\cap \mathcal{R}_{F_{i+1}}$.

Note that $R_0$ must be adjacent (share the common edge) with $R$.
Otherwise, both of vertices $M_0$ and $M_1$ belong to the same
component of $\mathcal{C}_{x,y}$, and we obtain a contradiction
with the assumption that the path described in (\ref{E:shorty}) is
minimal.

 In a similar way, we obtain that for any $i=1,2,\ldots,n$, the
region $R_i$ must be adjacent with at least one of regions $R,
R_0, R_1,\ldots,R_{i-1}$. If not, we have that the perfect
matching $\overline{M}=M_0\triangle R_i$ belongs to
$\mathcal{C}_{x,y}$, and $\overline{M}$ and $M'$ are contained in
the same component of $\mathcal{C}_{x,y}$. In that case we obtain
a contradiction, because the path
$$\overline{M}=\overline{M}_0 \frac{_{R_0}}{}
\overline{M}_1 \frac{\textrm{\,\,\,\, }}{}\ldots
\frac{\textrm{\,\,\,\, }}{}
 \overline{M}_{i-1} \frac{_{R_{i-1}}}{}
\overline{M}_{i+1} \frac{_{R_{i+1}}}{}\ldots
\frac{\textrm{\,\,\,\, }}{}\overline{M}_{n}
 \frac{_{R_{n}}}{}\overline{M}_{n+1}=M''$$
 is shorter than (\ref{E:shorty}). Here we let that
$\overline{M}_{j+1}=\overline{M}_j\triangle R_j$.

Let $e'$ denote a common edge of regions $R_0$ and $R$ that is
contained in $M'$. Note that $e'$ is not contained in $M_1$.
However, this edge is again contained in $M''$, and we conclude
that the region $R_0$ has to reappear again in (\ref{E:shorty}).

Let $R_{i_0}=R_0$ denote the first appearance of $R_0$ in
(\ref{E:shorty}) after the first step.
 There are the following
three possible situations that enable the reappearance of $R_0$:
\begin{enumerate}
    \item[$(a)$] All regions $R_k$ (for $k$ between $0$ and $i_0$) are disjoint
    with $R_0$.\\
In that case, we can omit the steps in (\ref{E:shorty}) labelled
by $R_0$ and $R_{i_0}$, and obtain a shorter path between $M'$ and
$M''$.

    \item[$(b)$] Any region that shares at least one edge with $R_0$
    appears an odd number of times between $R_0$ and $R_{i_0}$.
\\
This is impossible, because $R$ (that share an edge with $R_0$)
can not appear in (\ref{E:shorty}).

    \item[$(c)$] There is $t<i_0$
    such that $R_t=\bar{R}$ shares an edge with
$R_0$, but the fragment of the sequence (\ref{E:shorty})
    between $R_0$ and $R_{i_0}$ does not contain all
    region that shares an edge with $R_0$.
\\
Then the same region $\bar{R}$ has to appear again as $R_s$, for
some $s$ such that $t<s<i_0$.
 Again, if all regions $R_j$ are
disjoint with $\bar{R}$ (for $j=t+1,\ldots,s-1$), we can omit
$R_t$ and $R_s$, and obtain a contradiction. If not, there exist
indices $t'$ and $s'$ such that $t<t'<s'<s$ and $R_t'=R_s'$. We
continue in the same way, and from the finiteness of the path,
obtain a shorter path than (\ref{E:shorty}).
\end{enumerate}

 \qed

\textit{Continue of Proof:}
 We built $\mathcal{C}(G)$ by starting
with $\mathcal{C}(G\setminus e)$, that is a union of contractible
complexes by assumption. Then we glue the components of $
Prism(\mathcal{C}(G\setminus R))$ one by one.

After that, we glue all components of $\mathcal{C}(G\setminus
\{x,y\})$. At each step we are gluing two contractible complexes
along a contractible subcomplex, or we just add a new contractible
complex, disjoint with previously added components. From the
Gluing Lemma (see Lemma 10.3 in \cite{TM}) we obtain that
$\mathcal{C}(G)$ is contractible, or a disjoint union of
contractible complexes.

\end{proof}

\begin{rem}\label{R:whenCiscont}

 For a connected bipartite planar graph $G$ that
satisfy the condition $(**)$, the cubical matching complex
$\mathcal{C}(G)$ is collapsible,
 see
Theorem \ref{T:EHr}. The planar graph on the right side on Figure
\ref{F:primjerigener}
 satisfies the condition $(**)$, but
the corresponding cubical complex is not collapsible, it is a
union of three disjoint edges. So, there is a natural question:

 \textit{Is there a property of $G$ that provides the
collapsibility of its cubical complex  $\mathcal{C}(G)$?}
Obviously, if all complexes that appear on the right-hand side of
(\ref{E:RRforC}) are nonempty and contractible, then
$\mathcal{C}(G)$ is contractible.
\end{rem}

\section{The $f$-vector of domino tilings}

 The concept of tilings of a bipartite planar graph
generalizes the notion of domino tilings.
 Let $\mathcal{R}$ be a simple connected region, compound of unit squares in
the plane, that can be tiled with domino tiles $1\times 2$ and
$2\times 1$. The set of all tilings of $\mathcal{R}$ by domino
tiles and $2\times 2$ squares defines a cubical complex, denoted
by $\mathcal{C}(\mathcal{R})$. If we consider $\mathcal{R}$ as a
planar graph (all of its elementary regions are unit squares), and
if $G$ denotes the weak dual graph of $\mathcal{R}$ (the unit
squares of $\mathcal{R}$ are vertices of $G$), then
$\mathcal{C}(\mathcal{R})$ is isomorphic to the cubical matching
complex $\mathcal{C}(G)$, see Section 3 in \cite{EhrCUB} for
details. Note that the number of $i$-dimensional faces of
$\mathcal{C}(G)$ counts the number of tilings of $\mathcal{R}$
with exactly $i$ squares $2\times 2$.

Ehrenborg used collapsibility of $\mathcal{C}(G)$ to conclude (see
Corollary 3.1. in \cite{EhrCUB}) that the entries of $f$-vector of
$f(\mathcal{C}(G))=(f_0,f_1,\ldots,f_d)$ satisfy
\begin{equation}\label{E:linf}
 f_0-f_1+f_2-\cdots+(-1)^df_d=1.
\end{equation}

If $G$ is the weak dual graph of a region $\mathcal{R}$ that
admits a domino tiling, then all complexes that appear on the
right-hand side of the relation (\ref{E:RRforC}) are contractible
by induction, and therefore $\mathcal{C}(G)$ is contractible, see
Remark \ref{R:whenCiscont}. So, we obtain that the relation
(\ref{E:linf}) is true in any case, disregarding possible problems
with Theorem \ref{T:EHr}. In this Section we will prove that
(\ref{E:linf}) is the only linear relation for $f$-vectors of
cubical complexes of domino tilings.
\\

For all $n\in \mathbb{N}$, we let $G_n$ denote the following graph
$\begin{tabular}{|c|c|c|c|c|c|}
  \hline
  $1$ & $2$ & \,\,\, &\,\,\, &\,\,\,  & $n$ \\
  \hline
\end{tabular}.$

\noindent This graph (also known as the ladder graph) has $2n+2$
vertices, $3n+1$ edges and $n$ elementary regions (squares). For
$i=1,2,\ldots,n$, let $G_{n,i}$ denote the graph obtained by
adding one unit square below the $i$-th square of $G_n$. Now, we
describe some recursive relations for $f$-vectors of
$\mathcal{C}(G_{n})$ and $\mathcal{C}(G_{n,i})$.

\begin{prop}\label{P:rrforf}The entries of $f$-vectors of
 $\mathcal{C}(G_{n})$ and
$\mathcal{C}(G_{n,i})$ satisfy the following recurrences:
\begin{equation}\label{F:RRforf-vector}
f_i(\mathcal{C}(G_{n+2}))=f_i(\mathcal{C}(G_{n+1}))+
f_i(\mathcal{C}(G_{n}))+f_{i-1}(\mathcal{C}(G_{n})),
\end{equation}
\begin{equation}\label{F:RRforf-vectorleft}
    f_i(\mathcal{C}(G_{n+2,i}))=f_i(\mathcal{C}(G_{n+1,i}))+
f_i(\mathcal{C}(G_{n,i}))+f_{i-1}(\mathcal{C}(G_{n,i})),
\end{equation}
\begin{equation}\label{F:RRforf-vectorright}
f_i(\mathcal{C}(G_{n+2,i}))=f_i(\mathcal{C}(G_{n+1,i-1}))+
f_i(\mathcal{C}(G_{n,i-2}))+f_{i-1}(\mathcal{C}(G_{n,i-2})).
\end{equation}

\end{prop}

\begin{proof}
All formulas follow from relation (\ref{E:RRforC}), see the proof
of Theorem \ref{T:generalcase}. To obtain the formula
(\ref{F:RRforf-vector}), we apply (\ref{E:RRforC}) on $G_{n+2}$.
The rightmost vertical edge and the rightmost unit square in
$G_{n+2}$ act as $e$ and $R$ in (\ref{E:RRforC}).

\begin{figure}[htbp]
\begin{center}
\begin{tikzpicture}[scale=.491]{center}
\draw[ultra thick] (-4,-4) grid (1,-3);\node at (1.5,-3.5) {$=$};
\node at (7.5,-3.5) {$\sqcup$};\node at (13.5,-3.5) {$\sqcup$};
 \draw[ultra thick] (2,-4) grid (6,-3);
\draw[ultra thin] (2,-4) grid (7,-3); \draw[ultra thick] (7,-4) --
(7,-3); \draw[ultra thick] (12,-4) -- (13,-4); \draw[ ultra thick]
(12,-3) -- (13,-3);

  \draw[ultra thick]
(14,-4) grid (17,-3); \draw[ultra thin] (14,-4) grid (19,-3);
  \draw[ultra thick]
(18,-4) grid (19,-3); \draw[ultra thick] (8,-4) grid
(11,-3);\draw[ultra thin] (8,-4) grid (13,-3);


\draw[ultra thick] (-4,-1) grid (1,0);\node at (1.5,-.5) {$=$};
\node at (7.5,-.5) {$\sqcup$};\node at (13.5,-.5) {$\sqcup$};
 \draw[ultra thick] (2,-1) grid (6,0);
\draw[ultra thin] (2,-1) grid (7,0); \draw[ultra thick] (7,-1) --
(7,0); \draw[ultra thick] (12,-1) -- (13,-1); \draw[ ultra thick]
(12,0) -- (13,0);

  \draw[ultra thick]
(14,-1) grid (17,0); \draw[ultra thin] (14,-1) grid (19,0);
  \draw[ultra thick]
(18,-1) grid (19,0); \draw[ultra thick] (8,-1) grid
(11,0);\draw[ultra thin] (8,-1) grid (13,0);

 \draw[ultra thick] (-2,-7) grid (-1,-8);
 \draw[ultra thick] (16,-7) grid (17,-8);
  \draw[ultra thick] (10,-7) grid (11,-8);
   \draw[ultra thick] (4,-7) grid (5,-8);

\draw[ultra thick] (-4,-6) grid (1,-7);\node at (1.5,-6.5) {$=$};
\node at (7.5,-6.5) {$\sqcup$};\node at (13.5,-6.5) {$\sqcup$};
 \draw[ultra thick] (3,-7) grid (7,-6);
\draw[ultra thin] (2,-7) grid (7,-6); \draw[ultra thick] (2,-7) --
(2,-6); \draw[ultra thick] (8,-7) -- (9,-7); \draw[ ultra thick]
(8,-6) -- (9,-6);

  \draw[ultra thick]
(16,-7) grid (19,-6); \draw[ultra thin] (14,-7) grid (19,-6);
  \draw[ultra thick]
(14,-7) grid (15,-6); \draw[ultra thick] (10,-7) grid
(13,-6);\draw[ultra thin] (8,-7) grid (13,-6);
 \draw[ultra thick] (-2,-4) grid (-1,-5);
 \draw[ultra thick] (-2,-6) grid (-1,-7);
 \draw[ultra thick] (16,-4) grid (17,-5);
  \draw[ultra thick] (10,-4) grid (11,-5);
   \draw[ultra thick] (4,-4) grid (5,-5);

\node at (-6,-.5) {$(5)$};\node at (-6,-3.5) {$(6)$};\node at
(-6,-6.5) {$(7)$};
\end{tikzpicture}
\end{center}
\caption{The ``geometric proof`` of recursive relations for
$f(\mathcal{C}(G_{n}))$ and $f(\mathcal{C}(G_{n,i}))$.}
\label{Figure:rrfor}
\end{figure}
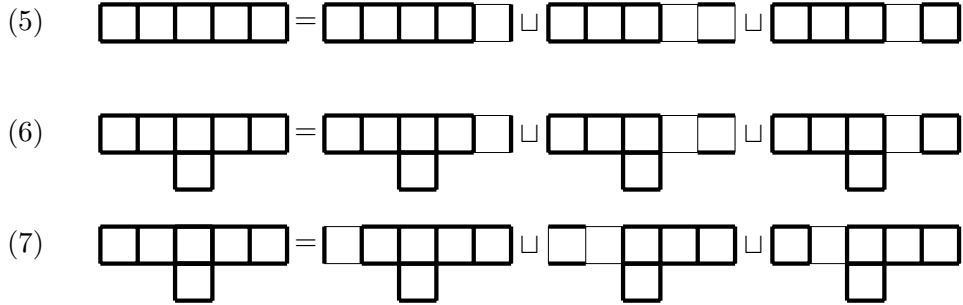

In the same way we can prove the remaining two relations. For each
relation, we choose an adequate elementary region $R$, a
corresponding edge $e$ of $R$, and use relation (\ref{E:RRforC}),
see Figure \ref{Figure:rrfor}.

\end{proof}

\noindent The $f$-vector
$(f_0,f_1,f_2,\ldots,f_{\lceil\frac{n}{2}\rceil})$ of
 $\mathcal{C}(G_{n})$ can be encoded by the polynomial
 $F_n$:
$$F_n=F_{\mathcal{C}(G_{n})}(x)=f_0+f_1x+f_2x^2+\cdots
+f_{\lceil\frac{n}{2}\rceil}x^{\lceil\frac{n}{2}\rceil}.$$
Similarly, we define the polynomials $F_{n,i}$ to encode the
$f$-vector of $\mathcal{C}(G_{n,i})$. Directly from
(\ref{F:RRforf-vector}) and (\ref{F:RRforf-vectorleft}) we obtain
that
$$F_{n+2}(x)=F_{n+1}(x)+(x+1)F_{n}(x),\hspace{.2cm}
 F_{n+2,i}(x)=F_{n+1,i}(x)+(x+1)F_{n,i}(x).$$
 Now, we define new polynomials $P_n$ and $P_{n,i}$ by
 $$P_n=P_n(x)=F_n(x-1),\hspace{.4cm}
  P_{n,i}=P_{n,i}(x)=F_{n,i}(x-1).$$ This is a variant of
  $h$-polynomial associated to corresponding
 cubical complexes.\\ From Proposition \ref{P:rrforf} it
 follows that the polynomials $P_n$ and $P_{n,i}$ satisfy
 the following recurrences

 \begin{equation}\label{F:RRforh-vector nolmal}
 P_{n+2}(x)=P_{n+1}(x)+xP_{n}(x),
\end{equation}
\begin{equation}\label{FFrrn_RIGHT}
P_{n+2,i}(x)=P_{n+1,i}(x)+xP_{n,i}(x),
\end{equation}
\begin{equation}\label{F:RRforh-vector_LEFT}
P_{n+2,i}(x)=P_{n+1,i-1}(x)+xP_{n,i-2}(x).
\end{equation}

\begin{rem}\label{R:exact}
We can use (\ref{F:RRforh-vector nolmal}) to obtain the
polynomials $P_n$ explicitly
$$P_{2d-1}={d\choose d}x^d+\cdots+{d+k\choose
d-k}x^k+\cdots+{2d-1\choose 1}x+{2d \choose 0}\textrm{, and}
$$
$$P_{2d}={d+1\choose d}x^d+\cdots+{d+k+1\choose
d-k}x^k+\cdots+{2d\choose 1}x+{2d+1 \choose 0}.
$$
\vspace{.2cm}

Note that the polynomials $P_n$ are related with Fibonacci
polynomials, see Section 9.4 in \cite{Benj-Quinn} for the
definition and a combinatorial interpretation of coefficients. The
coefficient of these polynomials are positive integers and the sum
of coefficients of $P_n$ is a Fibonacci number. Note that this is
just the number of vertices in $\mathcal{C}(G_n)$.

Assume that we embedded $\mathcal{C}(G_n)$ into $n$-cube as in
Proposition \ref{P:cord}, so that the perfect matching
$M_0=\begin{tabular}{|c|c|c|c|c|c|} \hdashline
   &  &  &  &  &  \\
   \hdashline
\end{tabular}$  of $G_n$ is the vertex in the origin. Now, the
coefficient of $x^k$ in $P_n$ counts the number of vertices of
$\mathcal{C}(G_n)$ for which the sum of coordinates is $k$, i.e.,
it is the number of vertices of $\mathcal{C}(G_n)$ whose distance
from $M_0$ is $k$.

Also, following \cite{Benj-Quinn}, we can recognize the
coefficient of $x^k$ in $P_n$ as the number of $k$-element subsets
of $[n]$ that do not contain two consecutive integers. Similarly,
we can interpret the coefficient of $x^k$ in $P_{n,i}$ as the
number of $k$-element subsets of the multiset
$M=\{1,2,\ldots,i-1,i,i, i+1,\ldots,n\}$ that do not contain two
consecutive integers. Note that the multiplicity of $i$ in $M$ is
two, and all other elements have the multiplicity one.

\end{rem}

\begin{defn}\label{D:opA}
Let $\mathcal{P}^d$ denote the vector
 space of all polynomials of degree at
 most $d$. We define the linear map
$A_d: \mathcal{P}^d\rightarrow \mathcal{P}^{d+1}$ recursively by

\begin{equation}\label{F:Arekrel}
    A_d(x^k)=xA_{d-1}(x^{k-1}) \textrm{ for all }k >0,
\end{equation}
\begin{equation}\label{F:A1}
   A_0(1)=1+2x\textrm{ and }A_d(1)=P_{2d+1}-A_d(P_{2d-1}-1).
\end{equation}

\end{defn}

\begin{lem}\label{L:0}
For any non-negative integer $d$, we have that
$$A_d(P_{2d-1})=P_{2d+1},A_d(P_{2d})=P_{2d+2} \textrm{ and }
 A_{d+1}(P_{2d}) =P_{2d+2}.$$
\end{lem}
\begin{proof}
From (\ref{F:A1}) it follows that $A_d(P_{2d-1})=P_{2d+1}$. For
the proof of the second formula we use (\ref{F:RRforh-vector
nolmal}), (\ref{F:Arekrel}) and induction
$$A_d(P_{2d})=A_d(P_{2d-1}+xP_{2d-2})=
P_{2d+1}+xA_{d-1}(P_{2d-2})=P_{2d+1}+xP_{2d}=P_{2d+2}.$$

The last formula in this lemma follows from (\ref{F:RRforh-vector
nolmal}) and earlier proved formulas
$$A_{d+1}(P_{2d})=A_{d+1}(P_{2d+1}-xP_{2d-1})=$$
$$P_{2d+3}-xA_d(P_{2d-1})=P_{2d+3}-xP_{2d+1}=P_{2d+2}.$$

\end{proof}

\begin{lem}\label{L:1}
For all integers $i$ and $d$ such that $1\leq i\leq \lfloor\frac d
2 \rfloor$, the following holds:
$$A_d(P_{2d-1,i})=P_{2d+1,i}\textrm{ and }
 A_d(P_{2d,i})=P_{2d+2,i}.$$
\end{lem}

\begin{proof}
For $i=1$ and $i=2$ we apply relation (\ref{E:RRforC}) in a
similar way as in the proof of Proposition \ref{P:rrforf}. We just
delete the only square in the second row of $G_{n,1}$ and
$G_{n,2}$, and obtain that
$$P_{2d-1,1}=P_{2d-1}+xP_{2d-3}, P_{2d-1,2}=P_{2d-1}+xP_{2d-4}.$$
By using Lemma \ref{L:0}, we obtain that
$$A_d(P_{2d-1,1})=A_d(P_{2d-1}+xP_{2d-3})=
P_{2d+1}+xP_{2d-1}=P_{2d+1,1}\textrm{, and }$$
$$A_d(P_{2d-1,2})=A_d(P_{2d-1}+xP_{2d-4})=
P_{2d+1}+xA_{d-1}(P_{2d-4})=$$$$ =P_{2d+1}+xP_{2d-2}=P_{2d+1,2}.$$
In a similar way, we can prove that
$$ A_d(P_{2d,1})=P_{2d+2,1}, A_d(P_{2d,2})=P_{2d+2,2}.$$
Assume that the statement of this lemma is true for $P_{2d-1,j}$
and $P_{2d,j}$ when $j<i+1$. Now, we use
(\ref{F:RRforh-vector_LEFT}) and induction to calculate
$$A_d(P_{2d,i+1})=A_d(P_{2d-1,i}+xP_{2d-2,i-1})=
A_{d}(P_{2d-1,i})+xA_{d-1}(P_{2d-2,i-1})=$$
$$=P_{2d+1,i}+xP_{2d,i-1}=P_{2d+2,i+1}.$$
From (\ref{FFrrn_RIGHT}) we obtain that
$$A_d(P_{2d-1,i+1})=A_d(P_{2d,i+1}-xP_{2d-2,i+1})=
A_{d}(P_{2d,i+1})-xA_{d-1}(P_{2d-2,i+1})=$$
$$=P_{2d+2,i+1}-xP_{2d,i+1}=P_{2d+1,i+1}.$$

\end{proof}

From Definition \ref{D:opA} and Remark
 \ref{R:exact} we can obtain
the concrete formula for the linear map $A_d$.
\begin{prop}\label{P:tacnoA_k}
For all $d,k \in \mathbb{N}$ such that $d\geq k\geq
 1$, we have that:
$$A_d(x^k)=x^k\left(1+2x-x^2+
2x^3-5x^4+14x^5-\cdots+(-1)^{d-k}C_{d-k}x^{d-k+1}\right).$$ Here
$C_m$ denotes the $m$-th Catalan number.
\end{prop}

\begin{proof}From (\ref{F:Arekrel}) it is enough to prove that
\begin{equation}\label{E:trueA_d}
A_d(1)=1+2x-x^2+2x^3-5x^4+\cdots+(-1)^dC_dx^{d+1}.
\end{equation} For all
integers $n$ and $k$ such that $n \geq k \geq 1$ (by using the
induction and the Pascal's Identity), we can obtain the next
relation
\begin{equation}\label{E:Catibin}
{n\choose k}=\sum_{i=0}^{k}(-1)^i{n+1+i\choose k-i}C_i.
\end{equation}
Now, we assume that (\ref{E:trueA_d}) is true for all positive
integers less than $d$, and calculate $A_d(1)$ by definition:
$$A_d(1)=P_{2d+1}-A_d(P_{2d-1}-1)=$$
$$=\sum_{i=0}^{d+1}{2d+2-i\choose i}x^i-
\sum_{i=1}^{d}{2d-i\choose i}x^{i}A_{d-i}(1).$$ The coefficients
of $1,x$ and $x^2$ in $A_d(1)$ are respectively:
$${2d+2\choose 0}=1,
{2d+1\choose 1}-{2d-1\choose 1}=2, {2d\choose 2}-{2d-2\choose
2}-2{2d-1\choose 1}=-1.$$ For $k>1$ the coefficient of $x^{k+1}$
in the polynomial $A_d(1)$ is
$${2d+1-k\choose k+1}-
{2d-k-1\choose k+1}-2{2d-k\choose
k}-\sum_{i=1}^{k-1}(-1)^{i}{2d-k+i\choose k-i}C_i.$$ From
(\ref{E:Catibin}) we obtain that the coefficient of $x^{k+1}$ in
$A_d(1)$ is $(-1)^{k}C_{k}$.

\end{proof}
\begin{cor}\label{C:Ais11}
For any positive integer $d$ the linear map $A_d$ is injective.
\end{cor}

\noindent Now, we consider all simple connected regions for which
the degree of the associated polynomial
$P_\mathcal{R}(x)=F_\mathcal{R}(x-1)$ is equal to $d$. Let
$\mathcal{F}^d$ denote the affine subspace of $\mathcal{P}^d$
spanned by these polynomials.
\begin{lem}\label{L:one is not}
The polynomial $P_{2d+1,d}$ is not contained in
$A_d(\mathcal{F}^d)$.
\end{lem}
\begin{proof} From (\ref{F:RRforh-vector_LEFT}) and (\ref{FFrrn_RIGHT})
we have that
$$P_{2d+1,d}-P_{2d+1,d-1}=(P_{2d,d-1}+xP_{2d-1,d-2})-
(P_{2d,d-1}+xP_{2d-1,d-1})=$$
$$-x(P_{2d-1,d-1}-P_{2d-1,d-2})=(-1)^{d+1}(x^{d+1}+x^d).$$
We know that $P_{2d+1,d-1}=A_d(P_{2d-1,d-1})$. If there exists a
polynomial $p\in \mathcal{F}^d$ such that $A_d(p)=P_{2d+1,d}$ then
we obtain
$$x^{d+1}+x^d=\pm A_d(p-P_{2d-1,d-1}),$$
which is impossible from Proposition \ref{P:tacnoA_k}.

\end{proof}

\begin{thm}
The polynomials $P_{2d-1},P_{2d},P_{2d-1,1},\ldots,P_{2d-1,d-1}$
are affinely independent in $\mathcal{F}^d$.
\end{thm}

\begin{proof}
We use induction on the degree. Assume that $d$ polynomials
$P_{2d-3}$, $P_{2d-2}$, $P_{2d-3,1}$, $\ldots,P_{2d-3,d-2}$ are
affinely independent in $\mathcal{F}^{d-1}$. From Lemmas \ref{L:0}
and \ref{L:1} and Corollary \ref{C:Ais11}, we conclude that
$P_{2d-1}$, $P_{2d}$, $P_{2d-1,1}$, $\ldots,P_{2d-1,d-2}$ are
affinely independent. These polynomials span a $(d-1)$-dimensional
affine subspace of $\mathcal{F}^d$. From Lemma \ref{L:one is not}
follows that $P_{2d-1,d-1}$ is not contained in
$A_{d-1}(\mathcal{F}^{d-1})$.

\end{proof}

\begin{cor}
The Euler-Poincare relation (\ref{E:linf}) is the only linear
relation for the $f$-vectors of tilings.
\end{cor}
This answer the question of Ehrenborg question about numerical
relations between the numbers of different types of tilings, see
\cite{EhrCUB}.

\bibliographystyle{amcjoucc}
\bibliography{amcexample}

\end{document}